\documentclass[a4,10pt,reqno,fleqn]{amsart}
\usepackage[none]{hyphenat}
\title[Closed ideals and Lie ideals of minimal tensor product]{Closed ideals and Lie ideals of minimal tensor product of certain C$^*$-algebras}
\author[B. Talwar]{Bharat Talwar}
\address{Department of
	Mathematics\\ University of Delhi\\ Delhi-110007, INDIA.}
\email{btalwar.math@gmail.com}

\author[R. Jain]{Ranjana Jain}
\address{Department of
	Mathematics\\ University of Delhi\\ Delhi-110007, INDIA.}
\email{rjain@maths.du.ac.in}

\thanks{The first named author was supported by a Junior Research Fellowship
	of CSIR with file number 09/045(1442)/2016-EMR-I}

\usepackage[all]{xy}
\usepackage{amsmath,amsfonts,amssymb,amsthm,hyperref,mathrsfs,setspace,verbatim,enumitem,cleveref}
\usepackage[margin=3.15cm]{geometry}
\newtheorem{theorem}{\bf Theorem}[section]
\newtheorem{lem}[theorem]{\bf Lemma}
\newtheorem{cor}[theorem]{\bf Corollary}
\newtheorem{remark}[theorem]{\bf Remark}
\newtheorem{prop}[theorem]{\bf Proposition}

\newtheorem{example}[theorem]{\bf Example}
\newcommand{\oop}{\widehat\otimes}
\newcommand{\seq}{\subseteq}
\newcommand{\oh}{\otimes^h}
\newcommand{\omin}{\otimes^{\min}}

\newcommand{\obp}{\otimes^\gamma}

\newcommand{\ot}{\otimes}
\newcommand{\ota}{\ot^{\alpha}}
\newcommand{\Z}{\mathcal{Z}}
\newcommand{\C}{\mathbb{C}}
\newcommand{\N}{\mathbb{N}}

\newcommand{\ol}{\overline}
\newcommand{\mcal}{\mathcal}
\newcommand{\CS}{$C\sp{\ast}$-algebra}
\newcommand{\CSS}{$C\sp{\ast}$-algebras}
\newcommand{\n}{\N_n}
\newcommand{\norm}[1]{\| #1 \|}
\begin{document}
\keywords{ Closed Lie ideal, \CS, minimal C$^*$-norm, tensor product}
\subjclass[2010]{46L06, 17B05}
\begin{abstract}
 For a locally compact Hausdorff space $X$  and a \CS \ $A$ with only finitely many  closed ideals, we discuss a characterization of closed ideals of $C_0(X,A) $ in terms of closed ideals of $A$ and certain (compatible) closed subspaces of $X$. We further use this result to prove that a closed ideal of $C_0(X) \omin A$ is a finite sum of product ideals. 
 We also establish that for a unital \CS \ $A$, $C_0(X,A)$ has centre-quotient property if and only if $A$ has 
centre-quotient property.
As an application, we characterize the closed Lie ideals of $C_0(X,A)$ and identify all closed Lie ideals of $ C_0(X) \omin B(H) $, $H$ being a separable Hilbert space.
\end{abstract}
\maketitle
\vspace{-4mm}
\section{Introduction}
Let $A\ot^\alpha B$ denote the completion of the algebraic tensor product of two C$^*$-algebras $A$ and $B$ under an algebra cross norm $\Vert\cdot\Vert_\alpha$. A natural question arises whether the closed ideals of $A \ota B$ can be identified in terms of the closed ideals of $A$ and $B$ or not. It is known that the closed ideals of the Banach algebras $A \oh B$, $A \obp B$ and $A \oop B$ are directly related to the closed ideals of $A$ and $B$, where $ \oh, \obp$ and $\oop$ are the Haagerup tensor product, Banach space projective tensor product and operator space projective tensor product, respectively. Interestingly, if either $A$ or $B$ possesses finitely many closed ideals then every closed ideal of these spaces is a finite sum of product ideals (see, for instance, \cite{sinc,vnr,knj}). However, in 1978,  Wassermann \cite{was} established an astonishing result that not every closed ideal of $B(H) \omin B(H)$ is a finite sum of product ideals, $\omin$ being the minimal $C^*$-tensor norm.
Recall that for every $x \in A \ot B$, $\| x \|_{\text{min}}= \text{sup} \{ \| (\pi_1 \ot \pi_2)(x)\| \}$, where $\pi_1$ and $\pi_2$ run over all representations of $A$ and $B$ respectively (see \cite{tak} for details).
In the present article, we prove that this anamoly can be removed by assuming one of the $C^*$-algebras to be commutative with the help of a totally different technique than those used for $\oh$, $\obp$ and $\oop$.

A Banach algebra  $B$ naturally imbibes a Lie algebra structure with the Lie bracket given by $[a,b] = ab-ba$ for every $a,b \in B$.
A closed subspace $L$ of  $B$ is said to be a \textit{Lie ideal} if $[B,L] \seq L$ where $[B,L] = \text{span} \{[b,l] : b \in B, l \in L  \}$. The closed Lie ideals for $C^*$-algebras are extensively studied, one may refer to the expository article  \cite{marcoux2010} for details. Recently some research has been done to identify the closed Lie ideals for the various tensor products of $C^*$-algebras. In \cite[Section 5]{rnv}, \cite[Section 4]{brv}) the closed Lie ideals of $C_0(X) \omin A$  have been characterized in terms of closed subspaces of $X$, $X$ being a locally compact Hausdorff space and $A$ being a simple $C^*$-algebra with at most one tracial state. However, if $A$ is not simple nothing is known about the closed Lie ideals of such spaces. In this article, we discuss a characterization of closed Lie ideals of $C_0(X,A)$, for any \CS \ $A$.

We present a  brief summary of the main results of the article. In Section 2, we first establish an appropriate (surjective) correspondence between a class of closed subspaces of $X$ and the closed ideals of $C_0(X,A)$, for any \CS \ $A$. Interestingly, this correspondence turns out to be bijective if $A$ has finitely many closed ideals. We use this correspondence to identify the image of a closed ideal of $C_0(X,A)$ in $C_0(X) \omin A$ under the canonical isomorphism. This will pave our way to establish  that every closed ideal of $C_0(X) \omin A$ is a finite sum of product ideals, where $X$ is a locally compact Hausdorff space and $A$ is a \CS \ with finitely many closed ideals.
In Section 3, we characterize the closed Lie ideals of $C_0(X,A)$ in terms of some closed subspaces of $X$.
To obtain a better picture of closed Lie ideals of $C_0(X) \omin A$, we prove that if $A$ is unital then the
centre-quotient property of $A$ passes to $C_0(X,A)$ and vice-versa.
As an application we establish an interesting result that a closed subspace $L$ of $C_0(X) \omin B(H)$ is a Lie ideal 
if and only if there exist closed subspaces $S_1,S_2$ of $X$ with $S_1 \seq S_2$ and a closed subspace $K$ of 
$C_0(X) \ot \C 1$ such that $L = \ol{J(S_1) \ot K(H) + J(S_2) \ot B(H) + K}$, where $H$ is a separable Hilbert space and 
for $F \subseteq X$,  $J(F) := \{ f \in C_0(X): f(F) \seq \{0\} \}$.

\section{Closed ideals of $C_0(X) \omin A$}
Let $X$ be a locally compact Hausdorff space and $A$ be any $C^*$-algebra. It is a well known fact that there is a bijective correspondence between the closed subspaces of $X$ and the closed ideals of $C_0(X)$ given by $F \leftrightarrow J(F)$. However, if we move from complex valued functions to the vector valued functions, such a correspondence is not known. Although, in the literature, it is established that every closed ideal of $C_0(X,A)$ is of the form $\{ f \in C_0(X,A): f(x) \in I_x, \, \forall x \in X \}$ where for every $x \in X$, $I_x$ is a closed ideal of $A$  \cite[V.26.2.1]{naim}, but this description fails to be fruitful while moving from $C_0(X,A)$ to $C_0(X)\omin A$ in order to determine the closed ideals.

We first generalize the former notion to the continuous vector valued functions by establishing a correspondence between the closed subspaces of X and closed ideals of $C_0(X,A)$. This correspondence will further enable us to characterize closed ideals of $C_0(X) \omin A$ in terms of closed ideals of $A$ and subspaces of $X$,  when $A$ has finitely many closed ideals. This is due to the fact that there exists an isometric $\ast$-isomorphism $\tilde{\varphi} : C_0(X) \omin A \rightarrow C_0(X,A)$, which takes $f \ot a$ to $af$ for every $f \in C_0(X)$ and $a \in A$, where $(af)(x)= f(x) a$ (see, \cite[Theorem 4.14 (iii)]{tak}, \cite[Proposition 1.5.6]{kani}).

Let us first fix some notations for further use. For any $t \in \N$, the set $\{1,2,3, \dots, t\}$ is be denoted by $\N_t$.
The spaces $C_b(X,A)$ and $C_c(X,A)$, as usual, denote the \CSS \ of all bounded continuous functions and compactly supported continuous functions, respectively, from $X$ to $A$ endowed with sup norm. For a non-unital \CS\ $A$, $\tilde{A}$ will denote its unitization.
For a locally compact Hausdorff space $X$ and any function $g \in C_0(X)$, we define $\hat{g} \in C_0(X,\tilde{A})$ (resp.,  $\hat{g} \in C_0(X,A)$)  by $\hat{g}(x) = g(x) 1$, where $1$ is the unit of $\tilde{A}$ (resp., of $A$) if $A$ is non unital  (resp., if $A$ is unital).
For any $a \in A$, we denote by $a'$ the constant function in $C_b(X,A)$ such that $a'(x)=a$ for every $x \in X$.

For an indexing set $\Delta$, let $S= \{ S_i \}_{i \in \Delta}$ and $T=\{ T_i \}_{i \in \Delta}$ be collections of subspaces of some sets $Y$ and $Z$.
We define \textit{$S$ to be compatible with $T$} if whenever for some subset $\gamma$ of $\Delta$, $\cap_{j \in \gamma} T_j = T_i$ for some $i \in \Delta$, then $\cap_{j \in \gamma} S_j = S_i$.
For a locally compact Hausdorff space $X$, $\alpha \in \Delta$ and $T$ as above, define a collection $\mcal{T}_{T}^{\alpha} := \{ S = \{S_i \}_{i \in \Delta} : S_i$ is a closed subspace of $X$ for every $i \in \Delta,\ S \, \text{is compatible with} \ T \ \text{and} \ S_{\alpha}=X \}$.
\begin{theorem}\label{infideal}
Let $X$ be a locally compact Hausdorff space, $A$ be a \CS \  and $\mcal{I} = \{I_i \}_{i \in \Delta}$ be the collection of all closed ideals of $A$ with $I_{\beta}= A$.
Then there exists a surjection $\theta$ from $ \mcal{T}_{\mcal{I}}^{\beta}$ into $\mcal{K}$, the set of all closed ideals of $C_0(X,A)$.
\end{theorem}
\begin{proof}
Define $\theta: \mcal{T}_{\mcal{I}}^{\beta} \rightarrow \mcal{K}$ by $\theta(S) = J(S)$ where $J(S) = \lbrace f \in C_0(X,A): f(S_i) \seq I_i, \ \forall i \in \Delta \rbrace$.
Clearly $\theta$ is well defined.
To see that $\theta$ is onto, consider $J \in \mcal{K}$.
For each $i \in \Delta$, set $S_i = \cap_{f \in J}{f^{-1}(I_i)}$.
If $I_i = \cap_{j \in \gamma} I_j$ for some subset $\gamma$ of $\Delta$, then $$S_i = \cap_{f \in J} f^{-1} (I_i) = \cap_{f \in J} f^{-1} (\cap_{j \in \gamma} I_j) = \cap_{f \in J}  \cap_{j \in \gamma} f^{-1} (I_j) =  \cap_{j \in \gamma} \cap_{f \in J}  f^{-1} (I_j) = \cap_{j \in \gamma} S_j,$$
so that $S =\{S_i\}_{i \in \Delta}$ is compatible with $\mcal{I}$. Since, $S_{\beta} = \cap_{f \in J} f^{-1} (I_{\beta}) = \cap_{f \in J} f^{-1} (A) = X$, we obtain $S \in \mcal{T}_{\mcal{I}}^{\beta}$. Clearly $J \subseteq J(S)$, so it is sufficient to prove that  $J$ is dense in $J(S)$.
For this, let $f \in J(S)$ be non-zero and $\epsilon > 0$ be arbitrary.

We first claim that for any $x \in X$, there exists an element $h_x \in J$ such that $\norm{ h_x(x) - f(x) } \leq \frac{\epsilon}{3} $.
Indeed, if $\gamma' = \lbrace i \in \Delta : x \in S_i \rbrace$, then $x \in \cap_{i \in \gamma'}S_i$ and $f(x) \in \cap_{i \in \gamma'}I_i = I_j$ for some $j \in \Delta$.
Let $I_r$ denote the closed ideal of $A$ generated by the set $\lbrace g(x) : g \in J \rbrace$.
Then $x \in S_r $ which implies that $f(x) \in I_r$. Hence $\norm{ \sum_{i=1}^{k} a_i g_i(x) b_i - f(x)} \leq \frac{1}{3} \epsilon$ for some $a_1, a_2, \dots a_k, b_1, b_2, \dots b_k$ in $A$ (resp., in $\tilde{A}$) if $A$ is unital  (resp., if $A$ is non unital). From \cite[Corollary 4.2.10]{knr}, we know that $J$ is a closed ideal of $C_0(X,\tilde{A})$, thus we have a function $h_x = \sum_{i=1}^{n} a_i' g_i b_i'$ in $J$ which satisfies the required condition. 

Denote by $\tilde{X}= X \cup \lbrace \infty \rbrace$, the one point compactification of $X$.
For each $x \in X$, let $\tilde{h}_x$ and $\tilde{f}$ be the continuous extensions of $h_x$ and $f$ to $\tilde{X}$ which take $\infty$ to $0$,  and let $\tilde{h}_{x_0}$ be the continuous extension of $h_{x_0}$, the zero function, to $\tilde{X}$ such that $\tilde{h}_{x_0}(\infty) =0$.
As $\tilde{h}_{x_0}(\infty) = \tilde{f}(\infty) = 0$, there exists a neighbourhood $V_{x_0}$ of $\infty$ in $\tilde{X}$ such that $\norm{\tilde{h}_{x_0}(y)-f(y) } \leq \frac{1}{3} \epsilon$ for every $y \in V_{x_0}$.  Further $\tilde{f}$ and $\tilde{h}_x$ are continuous at $x \in X$, so there exists a neighbourhood $V_x$ of $x$ such that  $\norm{h_x(y)-f(y)} \leq  \epsilon$ for every $y \in V_x$.

The collection $\lbrace V_{x_0} ,V_{x} \rbrace_{x \in \tilde{X}}$ is an open cover for $\tilde{X}$. 
As $\tilde{X}$ is compact, there is a finite subcover $\lbrace V_{x_0},V_{x_1}, V_{x_2} , \dots , V_{x_k}  \rbrace$  of $\tilde{X}$.
Then there exists a partition of unity $\lbrace g_0, g_1, g_2 ,  \dots g_k \rbrace$ subordinate to this subcover such that for every $i \in \lbrace 0,1,2, \dots ,k \rbrace$, $g_i \in C(\tilde{X})$, $0 \leq g_i \leq 1$, $\text{supp}(g_i) \seq V_{x_i}$ and for each $x \in \tilde{X}$, $\sum_{i=0}^{k}g_i(x) = 1$.
Now for $\tilde{h} =\sum_{i=0}^{k} \hat{g_i} \tilde{h}_{x_i} \in C(\tilde{X},A)$,  its restriction to $X$, given by $h =\sum_{i=0}^{k} \hat{g_i}_{|_X} h_{x_i}$, is an element in $J$. 
Hence for any $y \in X$, we have
\begin{eqnarray*}
\left \| \sum_{i=0}^{k} \hat{g_i}_{|_X}(y) h_{x_i}(y) - f(y)\right\| 
& = & \left \| \sum_{i=0}^{k}  \hat{g_i}(y) h_{x_i}(y) - \sum_{i=0}^{k} \hat{g_i}(y) f(y) \right\| \\ 
& \leq &  \sum_{i=0}^{k} \norm{ \hat{g_i}(y) } \norm{ f(y) - h_{x_i}(y)} \\ 
& \leq & \sum_{i=0}^{k}  g_i(y) \epsilon
\leq  \epsilon,  \\ 
\end{eqnarray*}
 thus $\norm{ h - f } \leq \epsilon $.
\end{proof}
\begin{remark}
	Note that the notation $J(S)$ given in the above theorem coincides with the previous notation of $J(F)$ by taking $S = \{F,X\}$ and $\mcal{I} = \{ \{ 0\}, A= \C \}$. 
\end{remark}

It is interesting to note that if $A$ possesses only finitely many closed ideals, then the above defined mapping $\theta$ turns out to be injective, hence providing a nice characterization of the closed ideals of $C_0(X,A)$. We would like to mention that the case when $A$ is a simple \CS \ is discussed in \cite[Proposition 4.1]{brv}, wherein we proved that every closed ideal of $C_0(X, A)$ is of the form
̃$\{f \in C_0(X, A):f(x) = 0,\, \forall x \in F \}$, for some closed subspace F of X.
\begin{theorem}\label{final}
Let $X$ be a locally compact Hausdorff space, $A$ be a \CS \ and $\mcal{I} = \{ I_1,I_2, \dots, I_n= A\}$,  $n \in \N$, be the set of all closed ideals of $A$.
Then the map $\theta$ from $\mcal{T}_{\mcal{I}}^{n}$ to the set of all closed ideals of $C_0(X,A)$ is a bijection. In particular, every closed ideal of $C_0(X,A)$ is precisely of the form $J(S) = \lbrace f \in C_0(X,A): f(S_i) \seq I_i, \ \forall i \in \n \}$, for a unique $S = \{S_i\}$ compatible with $\mcal{I}$.
\end{theorem}
\begin{proof}
Let $S = \{S_i \}_{i \in \n}$ and $S' =\{S_i' \}_{i \in \n}$ be two distinct elements of $\mcal{T}_{\mcal{I}}^{n}$ so that $S_i \neq S_i'$ for some $i \in \n \setminus \{n\}$.
Without loss of generality, we may assume that there exists $x \in S_i \setminus S_i'$.
Define a non-empty subset $\gamma = \lbrace j \in \n : x \notin S_j' \rbrace$ of $\n$.
  By Urysohn's Lemma \cite[Theorem 2.12]{rud} , as $V_x = (\cup_{i \in \gamma} S_i')^c $ is an open set containing $x$, there exists $g \in C_c(X)$ such that $g(x) = 1$, $g(X) \seq [0,1]$ and $supp(g) \seq V_x$.
It is easy to see that $(\cap_{j \in \gamma^c}I_j )\setminus  I_i \neq \emptyset $, because if $ \cap_{j \in \gamma^c}I_j  \seq  I_i $ then $ \cap_{j \in \gamma^c}S_j' \seq S_i'$, which is not true as $x \in \cap_{j \in \gamma^c}S_j'$ but $x \notin S_i'$.
Let $a \in (\cap_{j \in \gamma^c}I_j )\setminus  I_i$.
Consider the function $h = a'\hat{g} \in C_0(X,A)$.
Observe that $h \notin J(S)$, since $x \in S_i$ but $h(x) = a'(x) \hat{g}(x) = a \notin I_i$. However, we assert that $h \in J(S')$ which proves $J(S) \neq J(S')$. For $k \in \gamma$,  $y \in S_k'$ implies $y \in V_{x}^c$, so that $h(y)= a \hat{g}(y) = 0 \in I_k$. Also, if $k \in \gamma^c$, then $h(y) = a \hat{g}(y) \in (\cap_{j \in \gamma^c}I_j ) \seq I_k$ for every $y \in S_k'$. 
\end{proof}
In the quest of proving the main result regarding the characterization of closed ideals of $C_0(X) \omin A$, we require few more ingredients. 
\begin{lem}\label{tensorandfunction}
Let $X$ be a locally compact Hausdorff space and $A$ be a \CS.
Then for a closed subspace $C$ of $A$ and a closed ideal $J(Y)$ of $C_0(X)$, $Y \seq X$ being closed, there is an isometric isomorphism of Banach spaces 
$$ \ol{J(Y) \ot C}^{\min} \cong \lbrace f \in C_0(X,A): f(Y) = \lbrace 0 \rbrace , f(X) \seq C \rbrace.$$
\end{lem}
\begin{proof}
Denote by $J$ the closed subspace $\lbrace f \in C_0(X,A): f(Y) = \lbrace 0 \rbrace , f(X) \seq C \rbrace$ of $C_0(X,A)$, and $I=J(Y)$.
Let $\varphi$ denote the restriction of $\tilde{\varphi}$ to $\ol{I \ot C}^{\text{min}}$, where $\tilde{\phi}: C_0(X) \omin A \to C_0(X,A)$ is the isometric $*$-isomorphism as discussed earlier.
Then for $\sum_{j =1}^{n}f_j \ot c_j \in I \ot C$, we have  $\varphi(\sum_{j \in \n}f_j \ot c_j)(Y)  = \lbrace 0 \rbrace$, so that $\varphi (I \ot C) \seq J$.
Since $\varphi$ is an isometry, it is sufficient to prove that $\varphi(I \ot C)$ is dense in $J$.

Let $g \in J$ and $\epsilon > 0$ be arbitrary.
Since $J$ is also a closed subspace of $C_0(X,C)$ and $C_c(X,C)$ is dense in $C_0(X,C)$, there exists a function $h \in C_c(X,C)$ such that $\norm{{g - h}} < \epsilon/2 $.
Let $K:= \text{supp}(h)$, 
$B_r(b):=\{ c \in C : \norm{ c - b} < r \}$ and $B_r^{\times}(b):=B_r(b) \setminus \{0\}$, where $b \in C$ and $r>0$.
Since $\norm{h(y) } = \norm{g(y)-h(y)} < \epsilon/2 $ for every $y \in Y$, the collection $ \lbrace h^{-1}(B^{\times}_{\frac{\epsilon}{2} }(h(x)) \setminus \overline{h(Y)}): x \in K\setminus Y \rbrace$ $\cup  h^{-1}({ B_{\epsilon}(0))} $ forms an open cover of the compact set $K$. 
Fix a finite subcover, say, $ h^{-1}({B_{\epsilon}(0))}$ $ \cup \lbrace h^{-1}(
B^{\times}_{\frac{\epsilon}{2}}(h(x_i)) \setminus \overline{h(Y)})$ $: 1 \leq i \leq n \rbrace$.
Since $K$  is a compact subspace of a locally compact Hausdorff space $X$, there exists a partition of unity subordinate to this finite subcover, i.e. there exist functions $ f_0 , f_1, \dots, ,f_n$ in $C_c(X)$ such that $0 \leq f_i \leq 1$ for all $ 0 \leq i \leq n$, $\text{supp}(f_0) \seq U_0 := h^{-1}( B_{\epsilon}(0))$, $\text{supp}(f_i) \seq U_i := h^{-1}(
B^{\times}_{\frac{\epsilon}{2}}(h(x_i)) \setminus \overline{h(Y)})$ for all $1 \leq i \leq n$ and $\sum_{i = 0}^n f_i(x) = 1$  for $x \in K$(see \cite[Theorem 2.13]{rud}).

Let $ V=(\sum_{i = 0}^n f_i)^{-1}(0,3/2) $.
Then $V \cap (\cup_{i=0}^{n} U_i)$ is an open set containing $K$.
Pick $ \tilde{f} \in C_c(X)$ such that $\tilde{f}$ is  1 on $K$, $\text{supp}(\tilde{f}) \seq V \cap (\cup_{i=0}^{n} U_i)$ and $0 \leq \tilde{f} \leq 1 $.
Then for $ \tilde{f_i} = \tilde{f} f_i $, $\text{supp}(\tilde{f_i})  \seq V \cap U_i $ because $\text{supp}(f_i)  \seq U_i $ and $\text{supp}(\tilde{f})  \seq V $.
Now for  $x \in K$, we have $$\sum_{i = 0}^n \tilde{f_i}(x) = \sum_{i = 0}^n \tilde{f}(x) f_i(x) = \sum_{i = 0}^n f_i(x) = 1.$$
Also notice that $0 \leq \sum_{i = 0}^n \tilde{f_i} \leq 3/2 $ because for  $x \in V \cap (\cup_{i=0}^{n} U_i)$, $\sum_{i = 0}^n \tilde{f_i}(x) = \sum_{i = 0}^n \tilde{f}(x) f_i(x) \leq \sum_{i = 0}^n f_i(x) = 3/2$  and for $x \in (V \cap (\cup_{i=0}^{n} U_i))^c$, we have $\sum_{i = 0}^n \tilde{f_i}(x) = 0$ .

Now for $1 \leq i \leq n$, the open set $U_i$, and thus $V \cap U_i$ is disjoint from $Y$ so that $\sum_{i=1}^{n}  \tilde{f_i} \otimes h(x_i) \in I \ot C$. 
Fix  $x_0 \in K^c$, then for each $x \in X$
\vspace*{-2mm}
\begin{eqnarray*}
  \norm{ h(x) - \sum_{i = 1}^n \tilde{f_i}(x) h(x_i)} & =& \norm{ h(x) \sum_{i = 0}^n
  \tilde{f_i} (x) - \sum_{i = 0}^n \tilde{f_i} (x) h(x_i)} \\ & \leq & \sum_{i = 0}^n  \norm{ h(x)  - h(x_i)} \tilde{f_i} (x) \\
 & = & \sum_{i \ :\ x \in U_i \cap V} \norm{ h(x)  - h(x_i)} \tilde{f_i} (x) \ \ \ \  \ \ \quad (\text{since supp}(\tilde{f_i}) \seq U_i \cap V) \\
& <& \epsilon.
\end{eqnarray*}
Hence we obtain $\norm{ g - \varphi(\sum_{i=1}^{n} \tilde{f_i} \ot h(x_i))} < \frac{3 \epsilon}{2}$, proving that $\varphi(I \ot C)$ is dense in $J$.
\end{proof}

As a consequence of the above result, we have an interesting observation which identifies certain closed ideals of $C_0(X,A)$ with some closed ideals of $C_0(X) \omin A$.
\begin{cor}\label{idealfromset}
	Let $Y$ be a closed subspace of a locally compact Hausdorff space $X$. For any closed ideal $I$ of a \CS \ $A$, we have
	$$C_0(X) \omin I + J(Y) \omin A = \{f \in C_0(X,A): f(Y) \seq I \}$$ 
\end{cor}
\begin{proof}
Let $\mcal{I} = \{I_i \}_{i \in \Delta}$ be the set of all closed ideals of $A$ with $I_\beta = A$, for some $\beta \in \Delta$, and set $I = I_t$.
	If $t = \beta$ then the result is trivial.
	Otherwise, let $J_1 = C_0(X) \omin I_t + J(Y) \omin A$ and $J_2 = \{f \in C_0(X,A): f(Y) \seq I_t \}$.
	Then, by \Cref{infideal}, there exist elements $S = \{S_i \}_{i \in \Delta} $ and $S' = \{S_i' \}_{i \in \Delta}$ in $ \mcal{T}_{\mcal{I}}^{\beta}$ such that $J_1 = J(S)$ and $J_2 = J(S')$. It is sufficient to prove that $S = S'$.
	
	We first mention a common trick used in the proof.
	For any $x \in X$ and a closed subspace $F$ of $X$ with $x \notin F$, Urysohn's Lemma implies that there exists $f\in C_c(X)$ such that $f(x)= 1$ and $f(F)=0$.
	Then for any fixed $a \in A$ and any $y \in X$, there exists a function $g (y):= f(y) a $ in $C_0(X,A)$ such that $g(x)= a$ and $g$ vanishes on $F$.
	
	We now claim that $S_t = \cap_{f \in J_1} f^{-1}(I_t) = Y$.
	For $f \in J_1$,  $f = f_1 + f_2$ for some $f_1 \in C_0(X) \omin I_t$ and $f_2 \in J(Y) \omin A$. Thus, for any $y \in Y$,  $f(y) = f_1(y)+f_2(y) \in I_t$ as $f_2(y)=0$ by \Cref{tensorandfunction}, so that $Y \seq S_t$. 	For the reverse containment assume that $Y \subsetneq S_t$. Pick $a \in A \setminus I_t$ and $x \in S_t \setminus Y$, then there exists a function in $C_0(X,A)$ which vanishes on $Y$ and maps $x$ to $a$ which is a contradiction to the definition of $S_t$.
	On the similar lines, using the fact that $ S_t' = \cap_{g \in J_2} g^{-1}(I_t)$, one can easily deduce that $S_t =Y = S_t'$. 
	
	Now fix $i \in \Delta$ with $i \neq \beta, i \neq t$. Note that $J_2 = \cap_{i \in \Delta}\{f \in C_0(X,A): f(S_i') \seq I_i\}$, so that $G_t:= \{f \in C_0(X,A): f(S_t') \seq I_t \} \seq \{f \in C_0(X,A): f(S_i') \seq I_i \} = G_i$ (say), for every $i \in \Delta$.
		
	Case(i): $I_i \subsetneq I_t$, then $S_i' = \emptyset = S_i$.
	Because for $y \in S_i' \seq S_t'$ and $a \in I_t \setminus I_i$, there exists a function in $C_0(X,A)$ which takes $y$ to $a$.	Then such a function is in $G_t$ but not in $G_i$.	Also, if there exists an $x \in S_i \seq S_t = Y$, then $a'g$ as defined above will be a function in $C_0(X) \omin I_t \subset J_1$ which takes an element $x$ of $S_i$ to $a$ which does not belong to $I_i$, which is a contradiction to the definition of $S_i$.
	
	Case(ii): $I_t \subsetneq I_i$, then  $S_i= S_t = S_t' = S_i'$.
	To see this, if $S_t'$ is properly contained in $S_i'$, then for  $x \in S_i' \setminus S_t'$ and $a \notin I_i$, there exists a function in $C_0(X,A)$ which takes $x$ to $a$ and $S_t'$ to 0. This function belongs to $G_t$ but does not belong to $G_i$, which is a contradiction. Similarly, if $S_t$ is properly contained in $S_i$, then for $x \in S_i \setminus S_t$ and $a \notin I_i$, there is a function in $J(Y) \omin A \subset J_1$ which takes $x$ outside $I_i$ which contradicts the definition of $S_i$.

	Case(iii): $I_i$ is neither a subset nor a superset of $I_t$, then we claim that $S_i = S_i' = \emptyset$. If $I_j =I_t \cap I_i $, then $I_j\subsetneq I_t$ so that by Case(i), $S_t \cap S_i  = S_j = \emptyset=S_j' = S_t' \cap S_i'$.	Now, for $x \in S_i'$, $x \notin S_t'$, as argued in Case (ii), we obtain that $G_t$ is not contained in $G_i$, which is a contradiction, thus $S_i' =\emptyset$.
	Similarly, for $x \in S_i$, $x$ is not a member of $Y = S_t$. So for any $a \in I_t \setminus I_i$, applying the technique mentioned in the beginning, we get a function $g$ in $J(Y) \omin A \subseteq J_1$ such that $g(x) \notin I_i$, a contradiction.
	
	This proves that $S = S'$ and hence $J_1 = J_2$. 
\end{proof}

We are now ready to prove the main result of this section. Note that a \textit{product ideal} is a closed ideal of the form $I \omin J$, where $I$ and $J$ are closed ideals of $A$ and $B$, respectively. 
\begin{theorem}\label{fin sum}
Let $X$ be a locally compact Hausdorff space and $A$ be a \CS \, with finitely many closed ideals, say, $I_1,I_2, \dots, I_n$ with $I_1=\{0\}$ and $I_n= A$. Then for any closed ideal $K$ of $C_0(X) \omin A$, there exists $S = \{S_i \}_{i \in \n} \in \mcal{T}_{\mcal{I}}^{n}$, where $\mcal{I} =\{I_i\}_{i \in \n}$, such that
\begin{equation*}
\hspace{4cm} K = \sum_{j=2}^{n}  J(\cup_{k \in \gamma_j}S_k) \omin I_j,
\end{equation*}
where $ \gamma_j = \{ i \in \n : I_j \nsubseteq \ I_i \}$, for every $j \in \lbrace 2,3, \dots , n \rbrace$.

In particular, every closed ideal of $C_0(X) \omin A$ is a finite sum of product ideals.
\end{theorem}
\begin{proof}
By \Cref{final}, there exists $S = \{S_i \}_{i \in \n} \in \mcal{T}_{\mcal{I}}^{n}$ such that $K = J(S)= \lbrace f \in C_0(X,A) : f(S_i) \seq I_i, i \in \n \rbrace$. Set $K' = \sum_{j=2}^{n}  J(\cup_{k \in \gamma_j}S_k) \omin I_j $, then by \Cref{tensorandfunction}, $K'$ can be considered as a closed ideal of $C_0(X,A)$.
By virtue of \Cref{infideal}, it is sufficient to prove that $S_i= \cap_{f \in K'}f^{-1}(I_i)$ for every $i \in \n$.

It is clear that $ S_n  = X = \cap_{f \in K'}f^{-1}(A)$. Fix $i \in \N_{n-1}$ and consider any $x \in S_i$.
For $f \in K'$, $f= f_2 + f_3 + \dots + f_n$, where $f_r \in J(\cup_{k \in \gamma_r}S_k) \omin I_r  $ for every $r \in \lbrace 2,3, \dots , n \rbrace$.
Then for any such $r$, either $i \in \gamma_r$ or $i \in \gamma_r^c$.
If  $i \in \gamma_r$ then $f_r(x) = 0 \in I_i$.
If $i \in \gamma_r^c$ then $f_r(x) \in I_r \seq I_i$.
These two conclusions together imply that $S_i \seq \cap_{f \in K'}f^{-1}(I_i)$.

Next, pick $x \notin S_i$ and define $\alpha_i = \lbrace j \in \n : I_j \nsubseteq \ I_i \rbrace$. Note that $\alpha_i$ is non empty as $n \in \alpha_i$.
It is sufficient to prove the existence of a function $f \in K'$ such that $f(x) \notin I_i$.
We shall actually prove that such a function exists in the subset $\sum_{r \in \alpha_i}  J(\cup_{k \in \gamma_r}S_k) \omin I_r $ of $K'$. It is further enough to prove that there exists an $r \in \alpha_i$ such that $x \notin \cup_{k \in \gamma_r}S_k$, so that the required function $f$ exists in $J(\cup_{k \in \gamma_r}S_k) \omin I_r$.
In fact, by Urysohn's Lemma there exists a function $g \in C_c(X)$ such that $0 \leq g \leq 1$, $g(x) =1$ and $g(\cup_{k \in \gamma_r}S_k) = \{0\}$. Then by \Cref{tensorandfunction}, for $a \in I_r \setminus I_i$ (since $I_r \nsubseteq I_i)$,  the function $a' \hat{g}$ serves the purpose.
We claim that $\cap_{r \in \alpha_i} (\cup_{k \in \gamma_r}S_k) = S_i$, which will ensure the existence of such an $r$.

When $r \in \alpha_i$, we have $i \in \gamma_r$ and hence $S_i \seq \cap_{r \in \alpha_i} (\cup_{k \in \gamma_r}S_k)$.
We now prove the reverse inclusion.
Set $\alpha_i = \{r_1, r_2, \dots ,r_q \}$ and for each $r_j \in \alpha_i$, let there be $p_{r_j}$ number of elements in $\gamma_{r_j}$, say $\gamma_{r_j} = \{ r_{j,t_{r_j}} : 1 \leq t_{r_j} \leq p_{r_j} \}$.
So $$ \bigcap_{r \in \alpha_i} ( \cup_{k \in \gamma_r}S_k) = \bigcup_{1 \leq t_{r_j} \leq p_{r_j}} (S_{r_{1,t_{r_1}}} \cap S_{r_{2,t_{r_2}}} \cap \dots \cap S_{r_{q,t_{r_q}}}).$$

We have obtained that $ \cap_{r \in \alpha_j} ( \cup_{k \in \gamma_r}S_k)$ is a union of $\Pi_{j=1}^{q} p_{r_j}$ objects, each of which is an intersection of $q$ objects which looks like $S_{r_{1,t_{r_1}}} \cap S_{r_{2,t_{r_2}}} \cap \dots \cap S_{r_{q,t_{r_q}}}$.
Pick an ideal $I_{j_{r_1,t_{r_1}}} \cap I_{j_{r_2,t_{r_2}}} \cap \dots \cap I_{j_{r_q,t_{r_q}}}$.
Then there exists an $m \in \n$ such that $I_m = I_{r_{1,t_{r_1}}} \cap I_{r_{2,t_{r_2}}} \cap \dots \cap I_{r_{q,t_{r_q}}}$, and hence $S_m = S_{r_{1,t_{r_1}}} \cap S_{r_{2,t_{r_2}}} \cap \dots \cap S_{r_{q,t_{r_q}}}$.
If $S_m \seq S_i$, we are done.
Otherwise, $S_m \nsubseteq S_i$ will imply $I_m \nsubseteq I_i$ and hence $m \in \alpha_i$.
Thus $m = r_l$ for some $l \in \{ 1,2, \dots k \}$.
Then $r_{l,t_{r_l}} \in \gamma_{r_l}$ which implies $(I_m =) I_{r_l}  \nsubseteq I_{r_{l,t_{r_l}}}$, which is a contradiction to the fact that $I_m = I_{r_{1,t_{r_1}}} \cap I_{r_{2,t_{r_2}}} \cap \dots \cap I_{r_{q,t_{r_q}}}$.
\end{proof}
Let us demonstrate the above theorem with the help of few examples.
\begin{example}
	Let $H$ be a separable Hilbert space and $A = B(H) \oplus B(H)$.
	For $\mcal{I} = \{ I_1, \dots, I_9 = A\}$, the lattice of closed ideals of $A$ is given in the diagram below.

$$ 
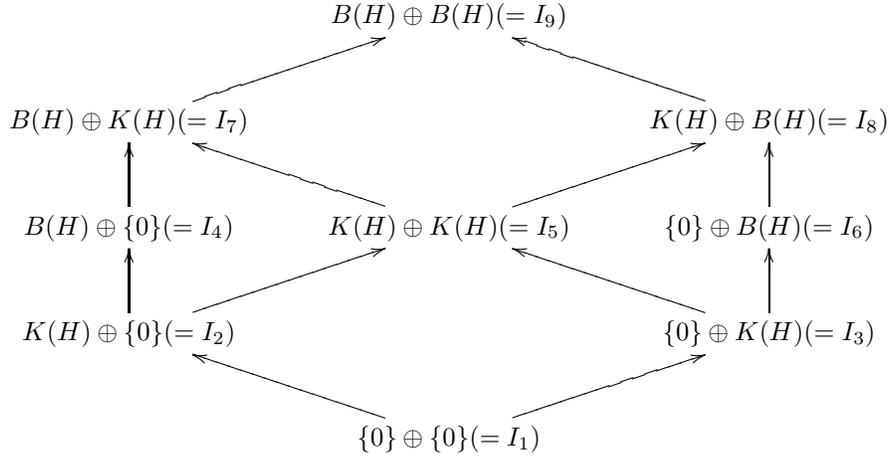
\xymatrix{
&B(H) \oplus B(H) (= I_9)\\
B(H) \oplus K(H) (= I_7) \ar[ru]  & \hfill & K(H) \oplus B(H) (= I_8) \ar[lu] \\	
B(H) \oplus \{0\} (= I_4) \ar[u]  & K(H) \oplus K(H) (= I_5) \ar[lu] \ar[ru] & \{0\} \oplus B(H) (= I_6) \ar[u]\\
K(H) \oplus \{0\} (= I_2) \ar[u] \ar[ru] & \hfill  & \{0\} \oplus K(H) (= I_3) \ar[u] \ar[lu] \\
&\{0\} \oplus \{0\} (= I_1) \ar[lu] \ar[ru] }$$

Looking at the diagram, we obtain that $\gamma_2 = \{1,3,6\}$, $\gamma_3 = \{1,2,4\}$, $\gamma_4 = \{1,2,3,5,6,8\}$, $\gamma_5 = \{1,2,3,4,6\}$, $\gamma_6 = \{1,2,3,4,5,7\}$, $\gamma_7 = \{1,2,3,4,5,6,8\}$, $\gamma_8 = \{1,2,3,4,5,6,7\}$ and $\gamma_9 = \{1,2,3,4,5,6,7,8\}$.
Hence, every closed ideal of $C_0(X) \omin A$ is of the form: 

\noindent
$J(S_6) \omin I_2 + J(S_4) \omin I_3 + J(S_8) \omin I_4 + J(S_4 \cup S_6) \omin I_5 + J(S_7) \omin I_6 + J(S_8) \omin I_7 + \\ J(S_7) \omin I_8 + J(S_7 \cup S_8) \omin I_9$, for a unique $S = (S_1, S_2, \dots ,S_9) \in \mcal{T}_{\mcal{I}}^{9}$, where a pictorial representation of $S$ is the following (here an arrow from $S_j$ to $S_k$ means that $S_j \subseteq S_k$, for $j,k \in \N_9$):
$$ 
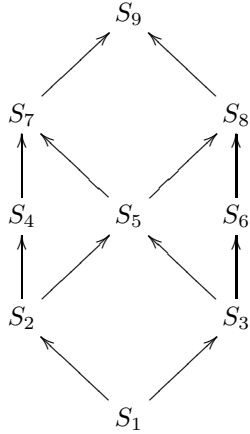
\xymatrix{
	&S_9 \\
	S_7 \ar[ru]  & \hfill & S_8 \ar[lu] \\	
	S_4 \ar[u]  & S_5 \ar[lu] \ar[ru] & S_6 \ar[u]\\
	S_2 \ar[u] \ar[ru] & \hfill  & S_3 \ar[u] \ar[lu] \\
	&S_1 \ar[lu] \ar[ru] }$$
\end{example}

We next  discuss the precise form of closed ideals of $C_0(X) \omin B(H)$, in terms of product ideals. Note that for a Hilbert space $H$, the set of all closed ideals forms a chain (see, \cite[Corollary 6.2]{luft}). In the following, $w_0$ denotes the cardinality of the set of all natural numbers and for every $i \in \N$, let  $w_i = 2^{w_{i-1}}$ so that $w_1$ is continuum.

\begin{example}\label{idealbh}
Let $X$ be a locally compact Hausdorff space and $H$ be a Hilbert space with  $w_n$ ($n \in \mathbb{N}$) as its Hilbert dimension.
Then the closed ideals of $C_0(X) \omin B(H)$ are of the form $\sum_{j=2}^{n+3} J(S_{j-1}) \omin I_j$ , where $I_j$'s are closed ideals of $B(H)$ and $S_j$'s are some closed subspaces of $X$.

To see this, let $\{0\} = I_1 \subsetneq I_2 \subsetneq I_3 \subsetneq \dots \subsetneq I_{n+3} = B(H) $ be the chain of closed ideals of $B(H)$ and $J$ be a closed ideal of $C_0(X) \omin B(H)$.
By \Cref{infideal}, there exists $n+3$ closed subspaces $ S_1 \seq S_2 \seq S_3 \seq \dots \seq S_{n+3} = X $ such that $J = J(S)$ where $S = \{S_i\}_{i=1}^{n+3}$.
As in \Cref{fin sum}, for $j \in \N_{n+3}$, $\cup_{k \in \gamma_j}S_k = S_{j-1}$ and hence $J = \sum_{j=2}^{n+3} J(S_{j-1}) \omin I_j$.
\end{example}

\section{Closed Lie ideals of $C_0(X,A)$}

The Lie normalizer of a subspace $S$ of a Lie algebra $A$ is defined by $N(S) := \lbrace a \in A : [a ,A] \seq S \rbrace$.
It can be easily verified that $N(I)$ is a closed subalgebra of $A$ for a closed ideal $I$ in $A$.
The Lie normalizer plays an important role in determining the Lie ideals of $A$ (for instance, see \cite{rnv},\cite{brv}). We identify the Lie normalizer of ideals of $C_0(X,A)$ and use this identification to characterize its closed Lie ideals.
\begin{theorem}\label{Lie}
	Let $X$ be a locally compact Hausdorff space and $A$ be a \CS \  with $\mcal{I} = \{I_i \}_{i \in \Delta}$ as the collection of all closed ideals such that $I_{\beta}= A$.
	Then a closed subspace $L$ of $C_0(X,A)$ is a closed Lie ideal if and only if there is an element $S = \{S_i \}_{i \in \Delta} \in \mcal{T}_{\mcal{I}}^{\beta}$ such that
	$$\lbrace f \in C_0(X,A) : f(S_i) \seq \ol{[I_i,A]}, \  \forall i \in \Delta    \rbrace \, \seq L \seq  \lbrace f \in C_0(X,A) : f(S_i) \seq N(I_i), \  \forall i \in \Delta    \rbrace.$$
\end{theorem}
\begin{proof}
	We know that a closed subspace $L$ of the \CS \ $C_0(X, A)$ is a Lie ideal if and only if there exists a closed ideal $J \seq C_0(X,A)$ such that $\ol{[J,C_0(X,A)]} \seq L \seq N(J)$ (\cite[Proposition 5.25, Theorem 5.27]{bks}).
	By \Cref{infideal}, $J = J(S)$ for some $S = \{S_i \}_{i \in \Delta} \in \mcal{T}_{\mcal{I}}^{\beta}$.
	Since for any fixed $a \in A$ and $x \in X$, there is an element in $C_0(X,A)$ which takes $x$ to $a$, we have
	\begin{eqnarray*}
		N(J) & = & \lbrace f \in C_0(X,A) : [f,g] \in J(S), \ \forall g \in C_0(X,A) \rbrace \\
		& = & \lbrace f \in C_0(X,A) : [f,g](x) \in I_i, \ \forall  \ x \in S_i, \, g \in C_0(X,A), \ \, \forall i \in \Delta   \rbrace \\
		& = & \lbrace f \in C_0(X,A) : f(x) \in N(I_i), \ \forall \ x \in S_i, \, i \in \Delta   \rbrace \\
		& = & \lbrace f \in C_0(X,A) : f(S_i) \seq N(I_i), \  \forall i \in \Delta   \rbrace.
	\end{eqnarray*}

	We know from \cite[Proposition 5.25]{bks} that $\ol{[I,B]} = I \cap \ol{[B,B]}$ for closed ideal $I$ of a \CS \ $B$.
	This fact, along with \Cref{tensorandfunction} gives
	\begin{eqnarray*}
		\ol{[J, C_0(X,A)]} 
		& = & \ol{[C_0(X,A),C_0(X,A)]} \cap J \\
		& = &   \ol{[C_0(X) \omin A,C_0(X) \omin A]} \cap J \\
		& = & \ol{C_0(X) \ot [A,A]} \cap J\\
		& = & \ol{C_0(X) \ot \ol{[A,A]}} \cap J \\
		& = & \lbrace f \in C_0(X,A): f(X) \seq \ol{[A,A]}\rbrace \cap \lbrace f \in C_0(X,A): f(S_i) \seq I_i, \ \forall i \in \Delta \rbrace \\
		& = & \lbrace f \in C_0(X,A): f(S_i) \seq \ol{[A,A]} \cap I_i, \ \forall i \in \Delta  \rbrace \\
		& = & \lbrace f \in C_0(X,A): f(S_i) \seq \ol{[I_i,A]}, \ \forall i \in \Delta  \rbrace.
	\end{eqnarray*}
	Hence the result.
\end{proof}

From the last result one can observe that if $N(I) = I +
\Z(A)$, $\mcal{Z}(A)$ being the centre of $A$, for every closed ideal $I$ of $A$, then for a closed ideal $J$ of 
$C_0(X,A)$, $N(J) = \lbrace f \in C_0(X,A) 
: f(S_i) \seq I_i + \Z(A), \  \forall i \in \Delta   \rbrace$ for some  $\{S_i \}_{i \in \Delta} \in 
\mcal{T}_{\mcal{I}}^{\beta}$.
The question that comes next is the following: Can we write every element of $N(J)$ as $g+h$ such that $g(S_i) \subseteq 
I_i$ for every $i \in \Delta$ and $h$ is $\Z(A)$-valued?
We shall  observe in \Cref{notrace} that a positive answer of this question will help us to obtain a better representation of all closed Lie ideals of $C_0(X,A)$.

Recall that a $C^*$-algebra $A$ is said to have the {\it centre-quotient property} if $\mcal{Z}(A/I) = (\mcal{Z}(A)+I) 
/I$ for every closed ideal $I$ of $A$ (see \cite{arc} for details). Using the fact that for the natural quotient map 
$\pi: A \rightarrow A/I$, $N(I) = \pi^{-1}(\mcal{Z}(A/I))$, we have a nice relation between Lie normalizer and the 
centre-quotient property.

\begin{lem}\label{idealpluscent}
	A \CS \ $A$ has centre-quotient property if and only if $N(I) = I + \mcal{Z}(A)$ for every closed ideal $I$ in $A$.
\end{lem}

A unital $C^*$-algebra $A$ is called {\it weakly central} if the continuous surjection $\psi : \text{Max}(A) \rightarrow \text{Max}(\mcal{Z}(A))$ given by $\psi(I) = I \cap \mcal{Z}(A)$ is an injection, where $\text{Max}(B)$ denotes the space of all  maximal ideals of a $C^*$-algebra $B$ endowed with the hull-kernel topology.
It is well known that a unital \CS \ with unique maximal ideal must have one dimensional centre \cite[Lemma 2.1]{arc}.
Since weak centrality and centre-quotient property are equivalent in unital \CSS \ (see, \cite[Theorem 1 and 2]{ves}),
presence of unique maximal ideal in a unital \CS \ $A$ implies that $A$ has centre-quotient property.
In \cite[Lemma 4.6]{brv} it was observed that for a simple unital \CS \ $A$ and a closed ideal $I \subseteq A$, $N(I) = I + C_0(X,\C 1)$.

\begin{theorem}\label{lienormalizer}
	Let $X$ be a locally compact Hausdorff space and $A$ be a unital \CS \ with unique maximal ideal.
	Then $N(J) = J + C_0(X,\C 1)$ for any closed ideal $J$ of $C_0(X,A)$.
\end{theorem}
\begin{proof}
	Note that $\mcal{Z}(C_0(X,A)) = C_0(X,\mcal{Z}(A)) = C_0(X, \C 1)$ 
	 and hence $J + C_0(X,\C 1) \seq N(J)$.
	Let  $\mcal{I} = \{I_i \}_{i \in \Delta}$ be the collection of all closed ideals of $A$ with $A = I_{\beta}$ and let $I_{\beta'}$ be the unique maximal ideal of $A$, $\beta, \beta' \in \Delta$.
	Then, \Cref{infideal}, there exists an element $S = \{S_i \}_{i \in \Delta} \in \mcal{T}_{\mcal{I}}^{\beta}$ such that $J = J(S)$.
	Since $A$ has centre-quotient property, by \Cref{idealpluscent}, $N(I_i) = I_i + \C 1$ for every $i \in \Delta$.
	Let $f \in N(J) =\lbrace g \in C_0(X,A) : g(S_i) \seq I_i + \C 1 ,\  \forall i \in \Delta   \rbrace$ as noted in \Cref{Lie}.
	Since $I_{\beta'}$ is the unique maximal ideal of $A$ and $S \in \mcal{T}_{\mcal{I}}^{\beta}$, we have $S_{\alpha} \seq S_{\beta'}$ for every $\alpha \in \Delta \setminus \{\beta\}$. 
	
	On $S_{\beta'}$, write $f = g+h$ which satisfy $h(S_{\beta'}) \subseteq \C 1$ and  $g(S_i) \seq I_i$ for every $i \in \Delta \setminus \{\beta\}$.
	This is possible as no proper ideal of $A$ can intersect $\C 1$.
	Since $I_{\beta'} \cap \ \C 1 = \lbrace 0 \rbrace$, by Hahn-Banach Theorem, there exists $T \in  A^{\ast}$ such that $\norm{T} =1$, $T(I_{\beta'}) = \lbrace 0 \rbrace$ and $T (\lambda 1) = \lambda$ for $\lambda \in \C$.
	Then $T f = T h$ on $S_{\beta'}$.
	Also $f$ vanishes at infinity and $\norm{T} =1$, so we obtain that $T f$ vanishes at infinity because $\norm{T f}  \leq \norm{f }$.
	Since $\norm{T h } = \norm{h} $, $T h$ and hence $h$ is a continuous function vanishing at infinity.
	So $g = f-h$ is continuous on $S_{\beta'}$ and is vanishing at infinity.
	By \cite[Theorem 4.5]{brv}, there exists an $h' \in C_0(X)$ such that $h'_{|_{S_{\beta '}}} = h$.
	For  $x \in S_{\beta'}^c$, define $g'(x) = f(x) - h'(x)$.
	Then $f = g'+h'$ with $g' \in J(S)$ and $h' \in C_0(X, \C 1)$ and we are done.
\end{proof}

It is observed in the previous result that if $A$ is a unital \CS \ with a unique maximal ideal then $C_0(X,A)$ has the centre-quotient property.
We now generalize this result and prove that the centre-quotient property of $A$ passes to $C_0(X,A)$. We would like to point out that the proof given above does not work when $A$ has more than one maximal ideal because in this case a closed ideal may intersect the centre non-trivially, that is, $I \cap \Z(A) \neq \{ 0\}$, and this is where the proof will fail.
As an intermediate step towards this generalization, we provide the following result.

\begin{prop} 
	Let $X$ be a compact Hausdorff space and $A$ be a unital  \CS \ having centre-
	quotient property. Then $C(X,A)$ has centre-quotient property.
\end{prop}
\begin{proof}
	It is sufficient to prove that $C(X,A)$ is weakly central. Consider maximal ideals $J_1$ and $J_2$ of $C(X,A)$ such that $J_1 \cap \Z(C(X,A)) = J_2 \cap \Z(C(X,A))$.
	For $i = 1,2$, there exist maximal ideals $I_i$ of $A$ and $x_i \in X$ such that $J_i = \{f \in C(X,A): f(x_i) \in I_i \}$ (\cite[Corollary V.26.2.2]{naim}). With the help of Urysohn's Lemma, one can easily verify that $I_i$ and $x_i$ are unique.  Since $\Z(C(X,A)) = C(X, \Z(A))$, we have  $J_i \cap \Z(C(X,A)) = \{ f \in C(X,Z(A)): f(x_i) \in I_i \cap \Z(A)  \}$. By the uniqueness, $x_1 = x_2$ and $I_1 \cap \Z(A) = I_2 \cap \Z(A)$. Since $A$ is weakly central, we obtain $I_1 = I_2$ and hence $J_1 = J_2$.
\end{proof}

\begin{theorem}\label{cqpinfunctionspace}
	If $X$ is a locally compact Hausdorff topological space and $A$ is a \CS. 
	If $C_0(X,A)$ has centre-quotient property then $A$ has centre-quotient property.
	Converse is true if $A$ is unital.
\end{theorem}
\begin{proof}
In order to prove that $A$ has centre-quotient property, it is sufficient to prove that for a closed ideal $I$ of $ A$, $N(I) \subseteq I + \mathcal{Z}(A)$. For a fixed $x \in X$, consider
a closed ideal $J$ of $C_0(X,A)$ given by $ \{ f \in  C_0(X,A) : f(x) \in I \}$. Now, for $a \in N(I)$, let  $f \in  C_0(X,A)$ such that $f(x) =a$.
	For any $g \in C_0(X,A)$, $(fg -gf)(x) = ag(x) - g(x) a  \in I$, which implies that $f \in N(J)$.
	Since $C_0(X,A)$ has centre-quotient property, $f = g + h$, for some $g \in J$ and $h \in C_0(X,\Z(A))$.
	Thus $a = f(x) = g(x) + h(x) \in I + \Z(A)$.
	
	Conversely, suppose that $A$ is unital and has centre-quotient property.
	If $\tilde{X}$ denotes the one point compactification of $X$, then we have a natural inclusion $C_0(X,A) \subseteq C(\tilde{X}, A)$.  For a closed ideal $J$ of $C_0(X,A)$, we claim that $ N(J)_{C_0(X,A)} \subseteq N(J)_{C(\tilde{X}, A)}$, where $ N(J)_{B}$ represents the Lie normalizer in $B$. 
Let $f  \in N(J)_{C_0(X,A)}$ and  $\{f_\mu\}$ be a quasi central approximate identity of the closed ideal $ C_0(X,A)$ of $C(\tilde{X}, A)$.
Then for any $g \in C(\tilde{X}, A)$, 
\begin{eqnarray*} 
 f g - gf  & = &  \lim f f_\mu g - \lim g f f_{\mu} \\
   &=&  \lim (f f_\mu g -  g f f_\mu + f_\mu gf - f_\mu gf) \\
   &=& \lim (f f_\mu g - f_\mu g f + f_\mu gf - gf f_\mu) \\
   &= & \lim (f f_\mu g - f_\mu g f ),
\end{eqnarray*}
since $\{f_\mu\}$ is quasi central approximate identity and $gf \in C_0(X,A)$.
Note that $f_\mu g \in C_0(X,A)$ and $f \in N(J)_{C_0(X,A)}$ together imply that $ ff_\mu g - f_\mu g f \in J$ for every $\mu$.
Thus $ fg - gf \in J$ which gives $f \in N(J)_{C(\tilde{X}, A)}$.

Since $C(\tilde{X}, A)$ has centre-quotient property (\Cref{cqpinfunctionspace}) and 
$J$ is also a closed ideal of $C(\tilde{X}, A)$, we have $N(J)_{C(\tilde{X}, A)} = J + C(\tilde{X}, \Z(A))$.
Now, for $f \in N(J)_{C_0(X,A)}$, there exist $g \in J$ and $h \in C(\tilde{X}, \Z(A))$ such that $f = g+h$.
If $\tilde{X} = X \cup \{\infty\}$, then $f(\infty) = 0 = g(\infty)$ which gives $h(\infty) = f(\infty) - g(\infty) = 0$ so that $h \in C_0(X,\Z(A))$. Thus  $N(J)_{C_0(X,A)} \subseteq J + C_0(X,\Z(A))$ and hence the result holds.
\end{proof}

We now characterize the Lie ideals of a class of $C^*$-algebras. Recall that a bounded linear functional $f$ on a \CS \ $A$ is said to be a \textit{tracial state} if  $f$ is positive of norm 1 and $f([a,b])= 0$ for every $a,b \in A$.

\begin{cor}\label{notrace}
	Let $X$ be a locally compact Hausdorff space and $A$ be a unital \CS \ with centre-quotient property and no tracial states.
	Then a closed subspace $L$ of $C_0(X,A)$ is a Lie ideal if and only if it is of the form $\ol{J+K}$ for some closed ideal $J$ of $C_0(X,A)$ and a closed subspace $K$ of $C_0(X, \Z(A))$.
\end{cor}
\begin{proof}
	Since $A$ has no tracial states, from \cite[Lemma 2.4]{brv}, $C_0(X,A)$ has no tracial states. Thus, by \cite[Proposition 5.25]{bks}, $\ol{[J,A]} = J$ for every closed ideal $J$ of $C_0(X,A)$.
	From  \cite[Theorem 5.27]{bks} and \Cref{cqpinfunctionspace}, a closed subspace $L$ of $C_0(X,A)$ is a Lie ideal if and only if  there exists a closed ideal $J$ of $C_0(X,A)$ such that $J \seq L \seq J+C_0(X,\Z(A))$.
	Hence $L$ must be of the form $\ol{J+K}$ for some closed subspace $K$ of  $C_0(X,\Z(A))$.
\end{proof}

As a consequence, we can now characterize all closed ideals of $C_0(X) \omin B(H)$.
\begin{cor}
	For a separable Hilbert space $H$ and a locally compact Hausdorff space $X$, a closed subspace $L$ of $C_0(X) \omin B(H)$ is a Lie ideal if and only if there exist two closed subspaces $S_1 \seq S_2$ of $X$ and a closed subspace $K$ of $C_0(X) \ot \C 1$ such that $$L = \ol{J(S_1) \ot K(H) + J(S_2) \ot B(H) + K}.$$
\end{cor}
\begin{proof}	
Since $B(H)$ has no 
tracial states and has a unique maximal ideal, the result is an easy consequence of \Cref{idealbh} and \Cref{notrace}.
\end{proof}

\end{document}